\documentclass[review,3p]{elsarticle}

\usepackage{lineno}

\usepackage{mathrsfs}

\modulolinenumbers[5]

\journal{Journal of Differential Equations}

\usepackage{amsmath}
\usepackage{amssymb}
\usepackage{amsthm}  

\newtheorem{thm}{Theorem}

\newtheorem{dfn}{Definition}

\newtheorem{problem}{Problem}

\begin{document}
\begin{frontmatter}

\title{Controllability of two-point boundary value problem for wave equations in $L^1$ and $L^2$ spaces: One dimensional case}

\author[mymainaddress]{Yuyou Gan}
\author[mymainaddress]{Sisi Huang\corref{cor1}}
\ead{sisih@zju.edu.cn}
\author[mymainaddress]{Dexing Kong\corref{cor1}}
\ead{dkong@zju.edu.cn}
\address[mymainaddress]{School of Mathematical Sciences, Zhejiang University, Hangzhou 310027, China}

\cortext[cor1]{Corresponding author.}
\fntext[fn1]{Declarations of interest: none.}
\fntext[fn2]{
Author contributions: Sisi Huang and Dexing Kong designed the study. Yuyou Gan and Sisi Huang wrote the paper. All the authors participated in the performance of the research. All the authors approved the final version for publication.
}
\fntext[fn3]{Funding: this research did not receive any specific grant from funding agencies in the public, commercial, or not-for-profit sectors.}

\begin{abstract}
In this paper we discuss the controllability of two-point boundary value problem (TBVP) for one-dimensional wave equation. Some new concepts are introduced: TBVP input control problem, minimum-input solution (MS) and pre-minimum-input solution (PMS). We set the metric in $L^1$ and $L^2$ spaces on a closed set, and control the input to reach its minimum. And we mainly discuss the property of input, the existence and uniqueness of MS and PMS for $L^1$ and $L^2$ metric respectively. The minimum inputs lie on a strip in $L^1$ and PMS for $L^1$ and $L^2$ always exists. Furthermore, to construct PMS, we also introduce an approximation method which meets certain conditions.
\end{abstract}


\end{frontmatter}

\section{Introduction}
Wave equation has applications in many fields, such as mathematical modeling of physical phenomena and even sociology problems \cite{wang2020global}. And control problem is one of today's most significant problem in science and technology, which describes to move the system from any given initial state to any other final state with an input \cite{kookos2004modern(p4-9)}. Moreover, in the field of distributed systems, exact controllability consists in trying
to drive the system to rest in a given finite time \cite{lions1988exact}. Kong has put forward two-point boundary value problem (TBVP) for differential equations \cite{dkong2010partial(p121-122)}. Consider the following first hyperbolic type, in other words, wave equation
\begin{equation}\label{1.1}
u_{tt}-c^2 \displaystyle\sum_{i=1}^n \frac{\partial^2 u}{\partial x_i^2}=0,
\end{equation}
where $t$ is the time variable, take $\boldsymbol x=(x_1, ..., x_n)$ are variables, $u = u(t, \boldsymbol x)$ is the unknown function of $t,\boldsymbol x$, and $c$ is some constant. So given two smooth functions $u_0(\boldsymbol x), u_T (\boldsymbol x)$ and a positive constant $T$, can we find a $C^2$-smooth function $u = u(t,\boldsymbol x)$ defined on $[0, T] \times \mathbb R^n$ such that the function $u = u(t, \boldsymbol x)$ satisfies the equation (\ref{1.1}) on the domain $[0, T] \times \mathbb R^n$, with the initial condition and the terminal condition
\begin{equation}\label{1.2}
u(0, \boldsymbol x) = u_0(\boldsymbol x),\ u(T, \boldsymbol x) = u_T (\boldsymbol x),\ \forall \boldsymbol x \in \mathbb R^n
\end{equation}
hold. This problem is specially put forward as TBVP of wave equation by Kong. Besides, Kong has also put forward TBVP for three dimensional wave equation, nonlinear wave equation, quasilinear wave equation and so on, which still remain open \cite{dkong2010partial(p121-122)}. 

In one-dimensional case, the existence of the solution of TBVP has been well discussed in \cite{kong2011two} and the solution is not unique. Based on this, we draw aspiration from minimum energy control \cite{Klamka2019controllability(p21-23)}, where it will bring the desired state with a minimum expenditure of energy. Therefore, another issue arises though our concepts of minimum energy are slightly different. What we actually desire to do here is to find the minimum input $u_t(0,x)$ for some metric $\mathscr{H}$ and the initial state $u(0,x)$ can be successfully transited to the final state $u(T,x)$. To be specific, find $u_t(0,x)$ such that
\begin{equation}
J=\|u_t(0,x)\|_{\mathscr{H}}
\end{equation}
reaches its minimum. And if the answer to the existence of minimum input is yes, what about its uniqueness? We call this input control problem.

In this article, we solve the input control problem of one-dimensional wave equation in $L^p\ (p=1,2)$ norm. The measure is confined on the tight interval $[-(2K_1+1)T,(2K_2+1)T],\ K_i \in \mathbb N^*$ can be arbitrarily large and thus the corresponding solution is defined on a trapezoidal region according to the wave propagation. Additionally, we further investigate the minimum input by bringing up some new concepts.

\subsection{Related work}

Kong \cite{dkong2010partial(p121-122), kong2011two} initiated TBVP for partial differential equations and discussed its exact controllability for several kinds of linear and nonlinear wave equations.  However, TBVP is essentially different from boundary control problem, which is more commonly known. There have been enormous study on boundary control problems for hyperbolic systems. Russell \cite{russell1978controllability} investigated the problem for linear partial differential equations and Lions \cite{lions1988exact} introduced a systematic method for exact controllability. Several other systems have also been discussed, with valuable results gained regarding to this field, which contains nonlinear hyperbolic systems, nonlinear wave equation, semilinear wave equations and so on (e.g., \cite{russell1973a, chewning1976controllability, lasiecka1989exact, lasiecka1989exact2, zuazua1990exact, zuazua1993exact}).

Another relevant field is minimum energy control problem, which is also different from input control problem. Minimum energy control problem is closely related to controllability problem, and was formulated and exhaustively discussed in Klamka \cite{Klamka2019controllability(p21-23)}.

Therefore, we can say results presented here are original. Additionally, we have established a special approximation method. As we all know, function approximation theory satisfying various spaces or needs is of significant value in the field of analysis and there have been many relevant classical theorems. Motivated by Bernstein polynomial and interpolation, we propose a $C^1$ approximation in $L^p[a,b]$ which meets certain requirements. Compared with classical theorems of function approximation, e.g., Bernstein polynomials, Stone–Weierstrass theorem, Chebyshev polynomials and proposition of approximation in $L^p$ \cite{folland1984real}, our approximation theorem can realize $C^1$ approximation in $L^p$ and at the same time: 
\begin{enumerate}
\item Adjust the difference between the values and derivative values at two endpoints.
\item Keep the function integral unchanged. 
\end{enumerate}
See section $2$ for more details.

\section{Problems and main results}

\begin{problem}{Two-point Boundary Value Problem (TBVP)}

Given two smooth function $f_0(x)$, $f_T(x)$ and a constant $T>0$, can we find a function $u(t,x)\in C^2([0,T]\times \mathbb R)$ such that
\begin{equation}\label{TBVP}
\left\{\begin{array}{l}{u_{tt}-c^2 u_{xx}=0,} \\ {u(0,x)=f_0(x),\ u(T,x)=f_T(x).}\end{array}\right.
\end{equation}
\end{problem}

In the following text, let $c=1$ without loss of generality. In fact let $\widetilde t=ct$, $u_{tt}-c^2 u_{xx}=0$ can be turned into $u_{tt}- u_{xx}=0$.

\begin{dfn}
$u(t,x)$ is the {\bf solution} if it satisfies TBVP equation (\ref{TBVP}). For a solution $u(t,x)$, call $v(x)=u_t(0,x)$ the  {\bf input} and all the inputs constitute a input space.
\end{dfn}

For another direction, the solution induced by input $v(x)$ is defined by D'Alembert equation
\begin{equation}
u(t,x) = \frac 1 2 (f_0(x+t)+f_0(x-t)) + \frac1 2 \int_{x-t}^{x+t}v(s)ds.
\end{equation}
It is easy to verify that the solution and input have one-to-one correspondence.

\begin{dfn}\label{dfn2}
Given some metric, a solution is a {\bf minimum-input solution (MS)} if its input is minimum in the input space, and a solution sequence is a {\bf pre-minimum-input solution (PMS)} if its corresponding input sequence converges to a function whose measure is the lower bound of input space for the given metric.
\end{dfn}

\begin{problem}{TBVP input control problem}

Given some metric to the input space of TBVP, the existence and uniqueness of minimum-input-solution (MS) and pre-minimum-input solution (PMS).
\end{problem}

The article discusses TBVP input control problem in $L^p[-(2K_1+1)T, (2K_2+1)T]\ (p=1,2)$, and $K_i\in \mathbb N^*,\ i=1,2$ could be arbitrarily large. In the following text, we may denote the metric space by $L^p$ for simplicity. In order to solve problem $2$ in $L^p$, what we mainly want to discuss is
\begin{equation}\label{eq.min}
min\ \displaystyle \int_{-(2K_1+1)T}^{(2K_2+1)T}|v(x)|^pdx,\ p=1,2.
\end{equation}
According to the wave propagation, the solution is in a {\bf trapezoidal region}
\begin{equation}
\Omega=\{(t,x):-(2K_1+1)T\leq x \leq (2K_2+1)T,\ 0\leq t\leq min\{-(2K_1+1)T+x,\ T,\ (2K_2+1)T-x \} \}.
\end{equation}

The following are some of the results we have obtained, and we will analyze and prove them step by step later.
\begin{enumerate}
\item What does the input look like and its properties.

The input $v(x)$ can be expressed as a function of the derivatives of $f_0(x)$ and $f_T(x)$, and $v(x),\ x\in [-(2K_1+1)T,(2K_2+1)T]$ can be determined by $v(x),\ x\in [-T,T)$. Additionally, the corresponding $u(t,x)$ is the $C^2$ solution if and  only if $v(x),\ x\in C^1[-T,T]$ satisfies integral condition and endpoints conditions. Then $[-T,T]$ is the \textbf{decision interval}. The value on the decision interval can be uniquely extended to $[-(2K_1+1)T,(2K_2+1)T]$ according to a certain recurrence relation, and the solution to the equation (\ref{TBVP}) is in $C^2(\Omega)$. As a result, with those certain conditions met, discussions on domain of $x$ can be changed from the original $ [-(2K_1 + 1) T, (2K_2 + 1) T] $ to $ [-T, T] $, and the domain of solution from a trapezoidal region to a triangle.

\begin{figure}[!htbp]
    \centerline{\includegraphics[width=12cm]{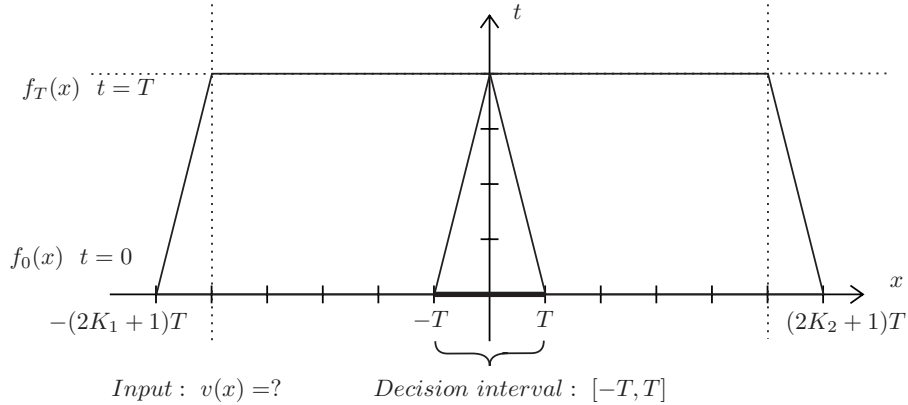}}
    \caption{Decision interval $[-T,T]$}
\end{figure}

\item The existence of MS.

In $L^1$, we can always find $v(x)\in C[-T,T]$ satisfying the integral condition, which makes the infimum of $\|v(x)\|_{L^1[-(2K_1+1)T,(2K_2+1)T]}$. Those $v(x)$ all lie in a strip region, which can be defined by order, or in some specific case it is a unique curve. However, the corresponding solution is not necessarily $C^2[-(2K_1+1)T,(2K_2+1)T]$. And MS exists if and only if there is a $v(x)\in C^1[-T,T]$ satisfying some endpoints condition as well. The result of $L^2$ is similar, we can also find $v(x)\in C^\infty[-T,T]$ satisfying the integral condition, and making the infimum of $\|v(x)\|_{L^2[-(2K_1+1)T,(2K_2+1)T]}$, and MS exists if and only if some endpoints conditions of $v$ on $[-T,T]$ are met.

\item The uniqueness of MS.

In $L^1$, if it is the specific case that such $v(x)$ described above is a unique curve satisfying certain conditions, then MS is also unique. Otherwise if MS exists then there should be infinite, but they all lie in the same order of strip region and MS is unique in this sense. In $L^2$, if MS exists, then it is unique.

\item The existence of PMS.

PMS of $L^1$ and $L^2$ always exist. Here we introduce a special approximation method which adjusts the endpoints value and keeps the integral unchanged.

\item The uniqueness of PMS.

The uniqueness of PMS is described by its limit. The limit of PMS in $L^1$ lies on the certain order of curve or strip almost everywhere. And PMS in $L^2$ converges to the original curve almost everywhere.
\end{enumerate}

\section{Preliminaries}\label{Preliminaries}

\begin{thm}\label{thm input}
In TBVP, given the recurrence relation
\begin{equation}
\label{recurrence relation}
v(x+T)=v(x-T)+2 f_{T}^{\prime}(x)-f_{0}^{\prime}(x+T)-f_{0}^{\prime}(x-T),
\end{equation}
the solution is $C^2$ if and only if $v(x)$ on $[-T,T]$ meets the following requirements:
\begin{enumerate}
\item $v(x)$ meets the integral condition
\begin{equation}
\label{integral}
\int_{-T}^{T} v(x) d x=2 f_{T}(0)-f_{0}(T)-f_{0}(-T).
\end{equation}
\item $v(x)$ is $C^1$ on $[-T,T]$.
\item The value and the derivative at the endpoints satisfy the following relations
\begin{equation}
\label{first derivative}
v(T)=v(-T)+2 f_{T}^{\prime}(0)-f_{0}^{\prime}(T)-f_{0}^{\prime}(-T)
\end{equation}
and
\begin{equation}
\label{second derivative}
v^{\prime}(T)=v^{\prime}(-T)+2 f_{T}^{\prime \prime}(0)-f_{0}^{\prime \prime}(T)-f_{0}^{\prime \prime}(-T).
\end{equation}
\end{enumerate} 
\end{thm}

\begin{proof}
According to D'Alembert equation,
\begin{equation}
\label{DAlembert}
u(t,x) = \frac 1 2 (f_0(x+t)+f_0(x-t)) + \frac1 2 \int_{x-t}^{x+t}v(s)ds.
\end{equation}
Let $t = T$ and get
\begin{equation}\label{3.6}
f_T(x) = \frac 1 2 (f_0(x+T)+f_0(x-T)) + \frac1 2 \int_{x-T}^{x+T}v(s)ds .
\end{equation}
Let $x=0$ in (\ref{3.6}) and the integral condition (\ref{integral}) follows. $u(t,x) \in C^2$ certainly implies $v(x) = u_t(0,x) \in C^1$. Then calculate the first and second derivatives of $u(t,x)$ with respect to $x$ at $x=0$ to get the equations (\ref{first derivative}) and (\ref{second derivative}).

For another direction, the recurrence relation combined with conditions 2 and 3 implies $v(x)\in C^1(\mathbb R)$. Define $u(t,x)$ as the equation (\ref{DAlembert}) and thus $u\in C^2$ follows because $f_0$ is a smooth function. Finally, we need to check the solution defined by (\ref{DAlembert}) satisfies TBVP. So it suffices to check
\begin{equation}
f_T(x)=\frac 1 2 (f_0(x+T)+f_0(x-T))+\frac 1 2 \int_{x-T}^{x+T} v(s) ds.
\end{equation}
From the recurrence relation (\ref{recurrence relation}), we get when $x\in [(2k-1)T, (2k+1)T],\ k\in \mathbb N^*,$
\begin{equation}
\label{v(x)=}
v(x)=v(x-2kT)-f_0'(x-2kT)+f_0'(x)+2\sum\limits_{n=0}^{k-1}f_T'(x-(2n+1)T)-2\sum\limits_{n=0}^{k-1}f_0'(x-2nT).
\end{equation}
And when $x\in [-(2k+1)T, -(2k-1)T],\ k\in \mathbb N^*,$
\begin{equation}
\label{v(x)=2}
v(x)=v(x+2kT)-f_0'(x+2kT)+f_0'(x)-2\sum\limits_{n=0}^{k-1}f_T'(x+(2n+1)T)+2\sum\limits_{n=0}^{k-1}f_0'(x+2(n+1)T).
\end{equation}
Without loss of generality, we just check the case when $x\in [(2k-1)T, (2k+1)T],\ k\in \mathbb N^*$,
\begin{equation}
\begin{array}{ll}
\displaystyle\int_{x-T}^{x+T}v(s)ds &= \displaystyle\int_{(2k-1)T}^{x+T} v(s) ds + \int_{x-T}^{(2k-1)T} v(s) ds\\ 
&= \displaystyle\int_{(2k-1)T}^{x+T} v(s) ds + \displaystyle\int_{x+T}^{(2k+1)T} v(s-2T) ds\\
&= \displaystyle\int_{(2k-1)T}^{x+T} v(s) ds + \displaystyle\int_{x+T}^{(2k+1)T} v(s)-2f_T'(s-T)+f_0'(s)+f_0'(s-2T) ds\\
&= \displaystyle\int_{(2k-1)T}^{(2k+1)T} v(s) ds + \displaystyle\int_{x+T}^{(2k+1)T}-2f_T'(s-T)+f_0'(s)+f_0'(s-2T) ds,
\end{array}
\end{equation}
Then it suffices to prove
\begin{equation}
\label{equilibrium}
2f_T(2kT)-f_0((2k+1)T) - f_0((2k-1)T) = \int_{(2k-1)T}^{(2k+1)T} v(s) ds.
\end{equation}
This equilibrium can be checked by substituting equation (\ref{v(x)=}) into the right hand side of (\ref{equilibrium}).
\end{proof}

\begin{dfn}
In a given metric space $X$, if $x,y\in X$ satisfy $\rho(x,y)<\epsilon$ in the metric space, then we say $x$ is $ \epsilon$-close to $y$ in $X$ , or $x$ and $y$ are $\epsilon$-close in $X$.
\end{dfn}

\begin{thm}\label{thm approximation}
For $1\leq p<\infty$ and arbitrarily given function $f(x)\in C[a,b]$, then for $\forall \epsilon >0$, there exists some $C^1$ function $g(x)$ s.t. $g(x)$ is $\epsilon-$close to $f(x) $ in $L^p[a,b]$ space, and satisfying the following conditions:
\begin{enumerate}
\item $g(b)=g(a)+c_1,g'(b)=g'(a)+c_2$, where $c_1,c_2$ are given constants.

\item $\displaystyle\int_a^b g(x) dx=\int_a^b f(x) dx$
\end{enumerate}
\end{thm}
\begin{proof}
Firstly let
\begin{equation}
g_1(x)=\left\{\begin{array}{ll}f(x),&x\in[a,b-\delta],\\ (f(a)+c_1-f(b-\delta) )/ \delta *(x-b+\delta)+f(b-\delta), &x\in(b-\delta,b],\end{array}\right.
\end{equation}
where $\delta>0$ is arbitrarily small. Because $f(x)\in C[a,b]$, then $\exists M>0$ s.t. $|f(x)|<M$. Let $M_1=max\{M,f(a)+c_1\}$, then
\begin{equation}
|\int_a^b g_1(x)-f(x)\ dx|\leq \int_a^b |g_1(x)-f(x)| dx\leq 2M_1 \delta.
\end{equation}
Let $r_1 = \displaystyle\int_a^b g_1(x)-f(x)\ dx$, and let $g_2(x)=g_1(x)-r_1/(b-a)$, then $g_2(x)\in C[a,b]$ satisfies both $g_2(b)=g_2(a)+c_1$ and $\displaystyle\int_a^b g_2(x) dx=\int_a^b f(x) dx$. Besides, we have
\begin{equation}
\begin{array}{ll}||g_2(x)-f(x)||_{L^p} &\leq ||g_2(x)-g_1(x)||_{L^p}+||g_1(x)-f(x)||_{L^p}\\&=\displaystyle\frac{|r_1|}{(b-a)^{1-1/p}}+(\int_{b-\delta}^b|g_1(x)-f(x)|^p\ dx)^{1/p}\\ & \displaystyle\leq\frac{2M_1 \delta}{(b-a)^{1-1/p}}+2M_1\delta ^{1/p} < \delta_1, \end{array}
\end{equation}
where $\delta$ can be arbitrarily small, so $\delta_1$ can also be arbitrarily small. Then $g_2(x)$ is $\delta_1-$close to $f(x)$ in $L^p$ space.

The Bernstein polynomial is used for approximation to construct the $ C^1 $ function on $ [a, b] $, i.e.,
\begin{equation}
B_n(g_2,x)=B_n(g_2)=\sum_{k=0}^n g_2(\frac k n(b-a)+a)C_n^k (\frac{x-a}{b-a})^k (1-\frac{x-a}{b-a})^{n-k}.
\end{equation}
Because $g_2(x)$ is a continuous function on $[a,b]$, then $B_n(g_2)$ uniformly converges to $g_2$ on $[a,b]$. Thus for arbitrarily small $\delta_2>0$, there exists $m\in \mathbb N^*$ s.t. 
\begin{equation}
|B_m(g_2,x)-g_2(x)|<\delta_2,\ \forall x \in [a,b].
\end{equation}
Let $g_3(x)=B_m(g_2,x)$, then $g_3(x)\in C^1[a,b]$. Thus
\begin{equation}
|\int_a^b g_3(x)-g_2(x)\ dx|\leq \int_a^b |g_3(x)-g_2(x)| dx\leq (b-a) \delta_2.
\end{equation}
Let $r_2 = \displaystyle\int_a^b g_3(x)-g_2(x)\ dx$ and $g_4(x)=g_3(x)-r_2/(b-a)$, then $g_4(x)\in C^1[a,b]$ and we have $\displaystyle\int_a^b g_4(x) dx=\int_a^b g_2(x) dx=\int_a^b f(x) dx$. And according to the construction of Bernstein polynomial, the value of the endpoints of $ g_3 (x) $ is consistent with $ g_2 (x) $, so
\begin{equation}
g_4(b)=g_4(a)+c_1.
\end{equation}
And we have
\begin{equation}
||g_4-f||_{L^p}\leq ||g_4-g_3||_{L^p}+||g_3-g_2||_{L^p}+||g_2-f||_{L^p}\leq \delta_2(b-a)^{1/p}+\delta_2(b-a)+\delta_1<\delta_3.
\end{equation}
Also $\delta_3>0$ can be arbitrarily small, $g_4(x)$ is $\delta_3$-close to $f(x)$ in $L^p$ space. In fact $g_4(x)$ is a smooth function. Let
\begin{equation}
g_5(x)=\left\{\begin{array}{ll}g_4(x),&x\in[a,b-\delta],\\H(x), &x\in(b-\delta,b],\end{array}\right.
\end{equation}
where $H(x)$ is a cubic Hermite spline and it ensures
\begin{equation}
H(b-\delta)=g_4(b-\delta),H'(b-\delta)=g_4'(b-\delta),H(b)=g_4(b),H'(b)=g_4'(a)+c_2.
\end{equation}
Specifically,
\begin{equation}
\begin{aligned} H(x)=& \frac {\delta^3-3\delta(x-b+\delta)^{2}+2(x-b+\delta)^{3}}{\delta^3} g_4(b-\delta) +\frac{(x-b+\delta)(x-b)^{2}}{\delta^2} g_4'(b-\delta) \\ &+\frac{3\delta\left(x-b+\delta\right)^{2}-2(x-b+\delta)^3}{\delta^3} g_4(b) +\frac{(x-b+\delta)^2 (x-b)}{\delta^2} (g_4'(a)+c_2), \end{aligned}
\end{equation}
when $x\in [b-\delta,b]$. The error of the estimate
\begin{equation}
|g_4(x)-H(x)|\leq \delta^4/384 * max_{b-\delta \leq x \leq b}\ g_4^{(4)}(x)\leq M_2 \delta^4\leq \delta_4.
\end{equation}
Also we let $\displaystyle r_3 = \int_a^b g_5(x)-g_4(x)\ dx$ and let $g(x)=g_5(x)-r_3/(b-a)$, then $g(x)\in C^1[a,b]$. And we have
\begin{equation}
\int_a^b g(x) dx=\int_a^b g_4(x) dx=\int_a^b f(x) dx.
\end{equation}
It satisfies the endpoint conditions
\begin{equation}
g(b)=g(a)+c_1,\  g'(b)=g'(a)+c_2.
\end{equation}
Finally, consider the distance between $g$ and $f$,
\begin{equation}
||g-f||_{L^p}\leq ||g-g_5||_{L^p}+||g_5-g_4||_{L^p}+||g_4-f||_{L^p}\leq \delta_4(b-a)^{1/p} +\delta_4 \delta^{1/p}+\delta_3.
\end{equation}
Then for $\forall \epsilon > 0$, because $\delta_4,\delta,\delta_3$ are arbitrarily small, thus we can make them sufficiently small to get $\delta_4(b-a)^{1/p} +\delta_4 \delta^{1/p}+\delta_3<\epsilon$, i.e.,
\begin{equation}
||g-f||_{L^p}<\epsilon.
\end{equation}
So $g(x)$ is $\epsilon$-close to $f(x) $ in $L^p[a,b]$.

\end{proof}

\begin{thm}\label{thm3}
On interval $[a,b]$, there is a decreasing sequence of continuous functions $+\infty>a_1(x)\geq a_2(x)\geq ...\geq a_k(x) > -\infty$, and let $a_0(x)=+\infty$, $a_{k+1}(x)=-\infty$. $U(h,x)$ is a binary function which holds following property: for every $(h,x)$ satisfying $a_{j+1}(x)<h<a_j(x)$ for some $j$, $\displaystyle\frac{\partial U}{\partial h} = c_j$ is a constant depending on $j$ and $\lbrace c_j \rbrace$ is a strictly decreasing sequence. Then
\begin{enumerate}
\item For every $A\in \mathbb R$, $\exists h(x)\in C[b,c],\ j\in\lbrace 0,...,k\rbrace$ such that 
\begin{equation}\label{ajh}
a_{j+1}(x)\leq h(x)\leq a_j(x),
\end{equation}
and
\begin{equation}
\displaystyle\int_b^c h(x) dx = A.
\end{equation}
\item $\displaystyle\int_b^c h(x) dx = A$ and $h(x)\in C[b,c]$, then
\begin{equation}\label{I[h]} 
I[h]=\displaystyle\int_b^c U(h(x),x)dx
\end{equation}
reaches its minimum if and only if $h$ satisfies equation (\ref{ajh}) for some $j$.
\end{enumerate}
\end{thm}
\begin{proof}
~\
\begin{enumerate}
\item Let 
\begin{equation}
p_1=\int_b^c a_j(x) dx \geq A,\ p_2=\int_b^c a_{j+1}(x) dx\leq A.
\end{equation}
When $j \ne 0$ and $k$. If $p_1=A$, take $h(x) = a_j(x)$. If $p_2=A$, take $h(x)=a_{j+1}(x)$. Else $p_1 > A > p_2$, let
\begin{equation}\label{hx1}
h(x)=\frac{A-p_2}{p_1-p_2} a_j(x)+\frac{p_1-A}{p_1-p_2} a_{j+1}(x).
\end{equation}
When $j = 0$, let
\begin{equation}\label{hx2}
h(x)=\frac{A}{p_2}a_1(x).
\end{equation}
When $j = k$, let
\begin{equation}\label{hx3}
h(x)=\frac{A}{p_1}a_k(x).
\end{equation}
We can see $h(x)\in C[b,c] $ and it is easy to check (\ref{hx1}), (\ref{hx2}), (\ref{hx3}) satisfies the integral and also $a_{j+1}(x)\leq h(x) \leq a_j(x)$.
\item $\exists j$ such that $p_1\geq A \geq p_2$. For any $h(x)$ such that $\displaystyle\int_b^c h(x) dx = A$ and $h(x)\in C[b,c]$, let
\begin{equation}
E_-=\lbrace x\in [b,c]|h(x)< a_{j+1}(x)\rbrace,\ 
E_+=\lbrace x\in [b,c]|h(x)\geq a_{j+1}(x)\rbrace,
\end{equation}
then
\begin{equation}
\begin{array}{ll}\label{EE}
I[h(x)] &=\displaystyle \int_b^c U(h(x),x) dx = \int_b^c (U(a_{j+1}(x),x)-\int_{h(x)}^{a_{j+1}(x)}\frac{\partial U(v,x)}{\partial v} dv)dx\\
&=\displaystyle \int_{E_-}(U(a_{j+1}(x),x)-\int_{h(x)}^{a_{j+1}(x)}\frac{\partial U}{\partial v}dv)dx + \int_{E_+}(U(a_{j+1}(x),x)-\int_{h(x)}^{a_{j+1}(x)}\frac{\partial U}{\partial v}dv)dx.
\end{array}
\end{equation}
In the above equation (\ref{EE}), when $x\in E_-$, we have $h(x)<a_{j+1}(x)$ and $\displaystyle\frac{\partial U}{\partial h}<c_j$. And thus 
\begin{equation}
\displaystyle\int_{h(x)}^{a_{j+1}(x)}\frac{\partial U(v,x)}{\partial v} dv \leq c_j(a_{j+1}(x)-h(x)).
\end{equation}
So we have
\begin{equation}\label{world}
\displaystyle \int_{E_-}(U(a_{j+1}(x),x)-\int_{h(x)}^{a_{j+1}(x)}\frac{\partial U}{\partial v}dv)dx\geq \int_{E_-} U(a_{j+1}(x),x)+c_j(h(x)-a_{j+1}(x))\ dx.
\end{equation}
Similarly, when $x\in E_+$, we have $h(x)\geq a_{j+1}(x)$ and $\displaystyle\frac{\partial U}{\partial h}\geq c_j$, so
\begin{equation}
\displaystyle\int_{a_{j+1}(x)}^{h(x)}\frac{\partial U(v,x)}{\partial v} dv \geq c_j(h(x)-a_{j+1}(x)).
\end{equation}
Thus we have
\begin{equation}
\begin{array}{ll}\label{hello}
\displaystyle\int_{E_+}(U(a_{j+1}(x),x)-\int_{h(x)}^{a_{j+1}(x)}\frac{\partial U}{\partial v}dv)dx&\displaystyle=\int_{E_+}(U(a_{j+1}(x),x)+\int_{a_{j+1}(x)}^{h(x)}\frac{\partial U}{\partial v}dv)dx\\&\displaystyle\geq \int_{E_+} U(a_{j+1}(x),x)+c_j(h(x)-a_{j+1}(x))\ dx.
\end{array}
\end{equation}
So, according to (\ref{world}) and (\ref{hello}), we get
\begin{equation}
\begin{array}{ll}\label{Ihx}
I[h(x)]&\geq \displaystyle \int_b^c U(a_{j+1}(x),x)+c_j(h(x)-a_{j+1}(x))\ dx\\
&=\displaystyle \int_b^c U(a_{j+1}(x),x)\ dx + c_j(A-p_2).
\end{array}
\end{equation}
Moreover, because $h(x)\in C[b,c]$, the equality sign in (\ref{Ihx}) holds if and only if $\displaystyle\frac{\partial U}{\partial h}\equiv c_j$, i.e., (\ref{I[h]}) reaches its minimum if and only if equation (\ref{ajh}) holds.
\end{enumerate}

\end{proof}

\begin{figure}[!htbp]
    \centerline{\includegraphics[width=16cm]{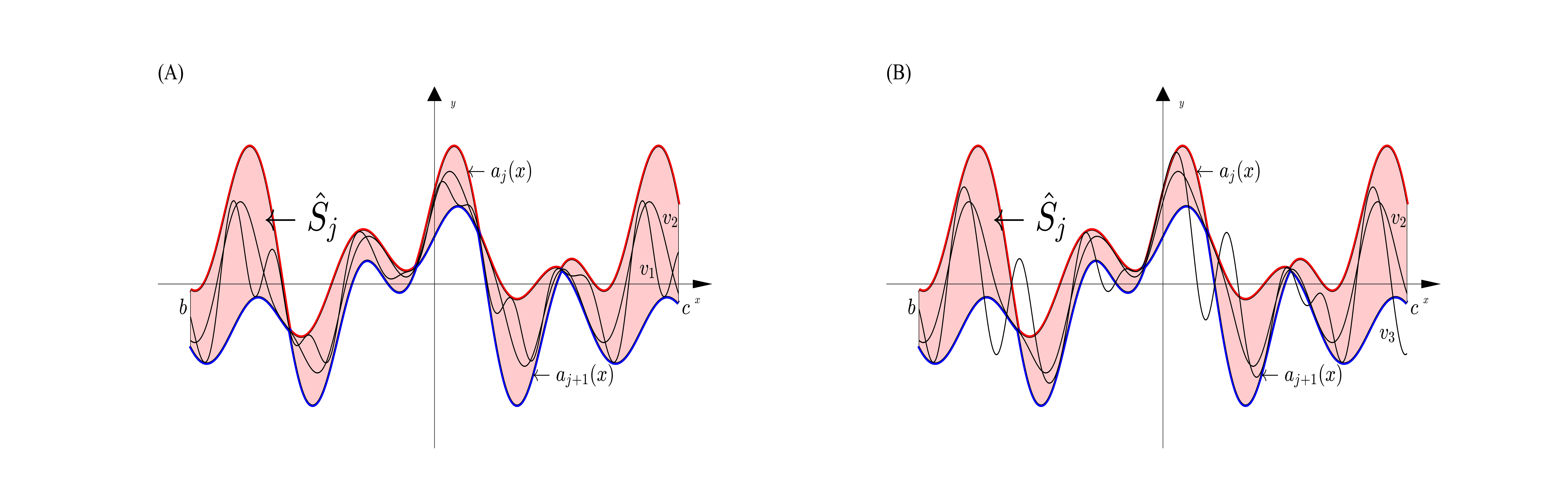}}
    \caption{Two situations of theorem $3$}
\end{figure}

Figure $2$ shows two situations of theorem $3$. $a_j(x)$ is the upper red line, $a_{j+1}(x)$ is the lower blue line, $a_{j+1}(x)\leq a_j(x)$. $\hat S_j$ is the red area between $a_j(x)$ and $a_{j+1}(x)$, $\displaystyle \int_b^c v_i(x) = A\ (i = 1,2,3)$, $\displaystyle \int_b^c a_{j+1}(x) dx < A < \int_b^c a_j(x) dx$. In figure 2(A), $v_1(x)$ and $v_2(x)$ both lie in $\hat S_j$, so $I[v_1(x)]=I[v_2(x)]$. In figure 2(B), $v_3(x)$ doesn't lie in $\hat S_j$, so theorem 3 tells $I[v_2(x)]<I[v_3(x)]$.

\section{Controllability of input in $L^p\ (p=1,2)$ spaces}

In TBVP, let $u(t,x)=F(x-t)+G(x+t)$, then
\begin{equation}
u(0,x)=F(x)+G(x)=f_0(x)\Rightarrow G(x)=f_0(x)-F(x).
\end{equation}
Thus
\begin{equation}
f_T(x)=u(T,x)=F(x-T)-F(x+T)+f_0(x+T),
\end{equation}
then
\begin{equation}\label{recurr}
\left\{\begin{array}{l}{F(x+2T)=F(x)-f_{T}(x+T)+f_{0}(x+2T)}, \\ {F(x-2T)=F(x)+f_{T}(x-T)-f_{0}(x)}.\end{array}\right.
\end{equation}
So the value of $F$ is determined by its value on $[-T,T]$. However, the value in this interval is not unique. Correspondingly, its solution is infinite.

For example, let $F\equiv 0$ on $[-T,T)$, then
\begin{equation}
F(x)=\left\{\begin{array}{l}{f_T(x+T)-f_0(x+2T)},\ x\in[-3T,-T), \\ {-f_T(x-T)+f_0(x),\ x\in [T,3T)}.\end{array}\right.
\end{equation}
In the same way, we can continuously find the value of $ F (x) $ on other $ 2T $ intervals according to the recursive formula (\ref{recurr}), and then substitute it into
\begin{equation}\label{eq2.0}
u(t,x)=F(x-t)-F(x+t)+f_0(x+t).
\end{equation}
We get the corresponding solution $u(t,x)$. Then
\begin{equation}
u_t(t,x)=-F'(x-t)-F'(x+t)+f_0'(x+t),
\end{equation}
and thus
\begin{equation}\label{eq.v}
v(x)=u_t(0,x)=-2F'(x)+f_0'(x).
\end{equation}
The thing we already know here is that when certain endpoint value conditions and some stronger requirements are met, the solution to the two-point boundary value problem of the one-dimensional wave equation exists and is not unique \cite{kong2011two}.

According to the actual situation as mentioned in the beginning, consider $x\in [-(2K_1+1)T,(2K_2+1)T]$. $F(x)$ satisfies the recurrence relation
\begin{equation}\label{eq.F.rec}
 F(x+T)=F(x-T)-f_T(x)+f_0(x+T),
\end{equation}
then from (\ref{eq.v}) and (\ref{eq.F.rec}) we get the recurrence relation of $v(x)$
\begin{equation}\label{eq.v.rec}
 v(x+T)=v(x-T)+2f_T'(x)-f_0'(x+T)-f_0'(x-T).
\end{equation}
So we know that the value of $ v (x) $ on $ [-T, T) $ determines its value on $ [-(2K_1 + 1) T, (2K_2 + 1) T]$. Moreover, with the recurrence relation (\ref{eq.v.rec}) and by theorem \ref{thm input}, the solution can be uniquely determined by $v(x)$ if and only if $v(x)\in C^1[-T,T]$ satisfies the integral condition 
\begin{equation}\label{integral.v}
\displaystyle\int_{-T}^{T} v(x) d x=2 f_{T}(0)-f_{0}(T)-f_{0}(-T)\triangleq A,
\end{equation}
as well as endpoints relations (\ref{first derivative}) and (\ref{second derivative}), which can be denoted as 
\begin{equation}\label{endpoints.v}
v(T)=v(-T)+c_1,\ v^{\prime}(T)=v^{\prime}(-T)+c_2.
\end{equation}
Therefore, it is reasonable to say $[-T,T]$ is the \textbf{decision interval} of $[-(2K_1+1)T,(2K_2+1)T]$.

And also from the above we can know that $v(x)$, $F'(x)$ are mutually determined. $F(x)=\displaystyle\int F'(x)+const$, by (\ref{eq2.0}) we can know that the constant term offsets and thus $u(t,x)$ is also uniquely determined. So $v(x), F'(x), u(t,x)$ can be found by two.

Furthermore, we calculate the input measured in $L^p[-(2K_1+1)T,(2K_2+1)T]$. Switch the bounds of integral interval by substituting equations (\ref{v(x)=}) and (\ref{v(x)=2}). Then $\|v(x)\|_{L^p}$ can be expressed by derivatives of $f_0$ and $f_T$. Specifically,
\begin{equation}\label{||v(x)||}
\begin{array}{ll}&\displaystyle\int_{-(2K_1+1)T}^{(2K_2+1)T}|v(x)|^p dx\\
&= \displaystyle\int^{T}_{-T}|v(x)|^p dx+\sum^{K_2}_{k = 1}\int^{(2k+1)T}_{(2k-1)T}|v(x)|^p dx+\sum_{k=1}^{K_1}\int_{(-2k-1)T}^{(-2k+1)T}|v(x)|^p dx\\
&= \displaystyle\int^{T}_{-T}|v(x)|^p dx+ \sum_{k=1}^{K_2}\int^{T}_{-T}|v(x)-t_{k+1}(x)|^p dx + \sum_{k=1}^{K_1}\int^{T}_{-T}|v(x)-t_{k+K+1}(x)|^p dx,
\end{array}
\end{equation}
where for $k=1,...,K_2$,
\begin{equation}\label{seq1t}
t_{k+1}(x) = f_0'(x)+2\sum_{n=0}^{k-1}f_0'(x+(2k-2n)T)-2\sum_{n = 0}^{k-1}f'_T(x+(2k-2n-1)T)-f'_0(x+2kT),
\end{equation}
and for $k=1,...,K_1$,
\begin{equation}\label{seq2t}
\displaystyle t_{k+K+1} = f_0'(x)+2\sum^{k-1}_{n=0}f'_T(x+(2n+1-2k)T)-2\sum_{n = 0}^{k-1}f'_0(x+2(n+1-k)T)-f'_0(x-2kT).
\end{equation}
Then denoting $t_1(x)\triangleq 0$, and let $K=K_1+K_2+1$, the above equation (\ref{||v(x)||}) is simplified to
\begin{equation}\label{Lp.input}
\displaystyle\int_{-(2K_1+1)T}^{(2K_2+1)T}|v(x)|^p dx
=\int^{T}_{-T}\sum_{i=1}^{K}|t_i(x)-v(x)|^pdx.
\end{equation}

\subsection{Results in $L^1$ space}

In (\ref{Lp.input}), let $p=1$,
\begin{equation}\label{input.L1}
\int_{-\left(2 K_{1}+1\right) T}^{\left(2 K_{2}+1\right) T}|v(x)| d x=\int_{-T}^{T} \sum_{i=1}^{K} |t_{i}(x)-v(x)| d x.
\end{equation}

\begin{dfn}
There is a sequence of functions $t_i(x),\ i=1,2,...,K$, then for some fixed $x$, the {\bf order} of $t_i(x)$ is its rank when putting them in a (not strictly) decreasing order and $Ord(t_i(x))=m$ means the rank is $m$. Moreover, if $t_i(x)\equiv m$ for $x\in (a,b)$, we say the order of $t_i(x)$ is well defined on the interval $(a,b)$ and denote the order by $Ord(t_i(x),x\in (a,b))=m$.
\end{dfn}

Now consider the sequence $t_i(x),\ i=1,...,K$ defined in (\ref{seq1t}) and (\ref{seq2t}). Therefore, for any fixed $x_0\in [-T,T]$ and for any $j\in\lbrace 1,...,K\rbrace$, we can always find some $t_k(x),\ k\in\lbrace 1,...,K\rbrace$ such that $Ord(t_k(x_0))=j$. Then for $x\in [-T,T]$, let $a_j(x)=t_k(x)$ such that $Ord(t_k(x))=j$. Then we can define a sequence of $a_j(x),\ j=1,...,K$ such that $Ord(a_j(x), x\in[-T,T])=j$.

Therefore, as is by the definition of order and the property of $t_i(x)\in C^{\infty}[-T,T] $, $\lbrace a_j(x)\rbrace$ is a decreasing sequence of continuous functions $+\infty > a_1(x)\geq	a_2(x) \geq ...\geq a_K(x)>-\infty$, and let $a_0(x)=+\infty$ and $a_{K+1}(x)=-\infty$. Then let 
\begin{equation}
U(v(x),x)=\sum_{i=1}^K|t_i(x)-v(x)|,\ x\in [-T,T],
\end{equation}
then
\begin{equation}
\frac{\partial U}{\partial v} =K-2j,\ a_{j+1}(x)\leq v\leq a_j(x),
\end{equation}
and it is left and right derivatives respectively when $v(x)=a_{j+1}(x)$ or $v(x)=a_j(x)$. And define
\begin{equation}
I[v]=\int_{-T}^T U(v(x),x) dx,
\end{equation}
where $\displaystyle \int_{-T}^T v(x) dx = A$.

As a result, TBVP input control problem in $L^1$ metric space can be answered as follows.

Firstly, according to theorem \ref{thm3}, $\exists v(x)\in C[-T,T],\ j$ such that 
\begin{equation}
a_{j+1} \leq v(x) \leq a_{j}(x).
\end{equation}
And all these $v(x)$ make up a set $S_j$, which lie in the {\bf strip region} between $a_{j+1}(x)$ and $a_j(x)$. Hereby denote the strip region containing the boundary as $\hat S_j$. Here we take index $j$ as the {\bf order of the strip}. Any $v(x)$ from $S_j$ minimizes (\ref{input.L1}). And if there is some $v(x)\in S_j$ is $C^1[-T,T]$, also satisfying endpoints equation (\ref{endpoints.v}), then its corresponding solution is MS of TBVP. 

Secondly, the MS is unique if and only if $\exists j$ such that $\displaystyle\int_{-T}^T a_j(x)dx=A$, and $a_j(x)$ is $C^1[-T,T]$ and satisfies endpoints equation (\ref{endpoints.v}). Otherwise the one-dimensional Lebesgue measure 
\begin{equation}
m(\lbrace x | v(x)\in (\hat S_j\backslash \partial \hat S_j) \rbrace) \neq 0.
\end{equation}
In this case we can always give the curve a little disturbance, then there should be infinite MS on the contrary. However, if we think inputs in the same order of strip $\hat S_j$ are equivalent, then the MS is unique in this sense.


Thirdly, by theorem \ref{thm approximation}, take any $v(x)\in S_j$, we can find a sequence of $v_n(x)$ such that 
\begin{equation}
\|v_n(x)-v(x)\|_{L^1[-T,T]}<\epsilon_n\xrightarrow{n} 0,
\end{equation}
with the following conditions satisfied
\begin{equation}
v_n(x)\in C^1[-T,T],\ \int_{-T}^T v_n(x) dx=\int_{-T}^T v(x) dx=A,\ v_n(T)=v_n(-T)+c_1,\ v'_n(T)=v'_n(-T)+c_2.
\end{equation}
And $\|v(x)\|_{L^1[-(2K_1+1)T,(2K_2+1)T]}\triangleq m,\ v(x) \in S_j$ is the lower bound of input space for the given metric. Note that
\begin{equation}
\begin{array}{ll}
\displaystyle |\int_{-\left(2 K_{1}+1\right) T}^{\left(2 K_{2}+1\right) T}|v_n(x)|  d x -m|
&\leq  \displaystyle \int_{-T}^{T} |\sum_{i=1}^{K}|t_{i}(x)-v_n(x)|-\sum_{i=1}^{K}|t_{i}(x)-v(x)|| d x \\
&\leq  \displaystyle \int_{-T}^{T} \sum_{i=1}^K|v(x)-v_n(x)| dx 
\leq 2KT\epsilon_n,
\end{array}
\end{equation}
so the corresponding solution sequence $\lbrace u_n\rbrace$ of the input sequence $\lbrace v_n(x)\rbrace$ is PMS according to definition \ref{dfn2}. If the input sequence $\lbrace v_n(x)\rbrace$ converges to $\tilde v(x)\in C^1[-T,T]$ satisfyting the endpoints relation (\ref{endpoints.v}), then the corresponding solution of the limit $\tilde v(x)$ must be MS, which should be contained in the statement of first point. In fact, as $v(x)\in C[-T,T]$, $\tilde v(x)$ is exactly $v(x)$. Otherwise the limit cannot induce MS but PMS always exists. For example, when $c_1$ in (\ref{endpoints.v}) does not belong to $[a_{j+1}(T)-a_{j}(-T),\ a_{j}(T)-a_{j+1}(-T)]$, then any function in $S_j$ can't meet $f(T) = f(-T)+c_1$, so MS does not exist in this case.

Lastly, the uniqueness of PMS can also be described in the sense of order of the curve $a_j$ or the strip $\hat S_j$ by considering its limit $\tilde v(x)$. If $v(x)=a_j(x)$, then $\tilde v(x)$ lies on $a_j(x)$ almost everywhere, that is one-dimensional Lebesgue measure 
\begin{equation}
m(\lbrace x|v(x)\neq a_j(x) \rbrace)=0. 
\end{equation}
Else if $a_{j+1} < v(x) < a_{j}(x)$, then $\tilde v(x)$ lies in the strip $\hat S_j$ almost everywhere, that is one-dimensional Lebesgue measure 
\begin{equation}
m(\lbrace x|v(x)\notin \hat S_j \rbrace)=0.
\end{equation}

\subsection{Results in $L^2$ space}

In (\ref{Lp.input}), let $p=2$,
\begin{equation}
\int_{-\left(2 K_{1}+1\right) T}^{\left(2 K_{2}+1\right) T}|v(x)| d x=\int_{-T}^{T} \sum_{i=1}^{K} |t_{i}(x)-v(x)|^2 d x,
\end{equation}
where $\displaystyle\int_{-T}^T v(x) dx=A$ and $v(x)\in C^1[-T,T]$. And by
\begin{equation}
\begin{array}{ll}&\displaystyle\int_{-T}^{T} \sum_{i=1}^{K}(t_{i}(x)-v(x))^2 dx\\=&\displaystyle K\int_{-T}^T(v(x)-\frac 1 K(\sum_{i=1}^K t_i(x)))^2dx+\int_{-T}^T (\sum_{i=1}^K t_i^2(x)-\frac 1 K (\sum _{i=1}^K t_i(x))^2) dx,\end{array}
\end{equation}
let $\displaystyle h(x)=v(x)-\frac 1 K(\sum_{i=1}^K t_i(x))$, and then it suffices to find
\begin{equation}
min \int_{-T}^T h^2(x) dx,
\end{equation}
such that
\begin{equation}
\int_{-T}^T h(x) dx=\int_{-T}^T v(x)-\frac 1 K(\sum_{i=1}^K t_i(x)) dx=A-\int_{-T}^T \frac 1 K(\sum_{i=1}^K t_i(x)) dx=A_1.
\end{equation}
By Holder's inequality,
\begin{equation}
|A_1|=|\int_{-T}^T h(x) dx|\leq (\int_{-T}^Th^2(x) dx)^{\frac 1 2}(\int_{-T}^T1^2 dx)^{\frac 1 2}=(\int_{-T}^Th^2(x) dx)^{\frac 1 2}\sqrt {2T},
\end{equation}
we get
\begin{equation}
\int_{-T}^{T}h^2(x)dx \geq \frac{A_1^2}{2T}.
\end{equation}
Since we need $h(x)$ to be continuous, to get the minimum value, in other words, the equal sign holds, if and only if $ h (x) $ is constant, and thus
\begin{equation}
h(x)=\frac{A_1}{2T},
\end{equation}
that is
\begin{equation}\label{vx.L2}
v(x)=\frac 1 K(\sum_{i=1}^K t_i(x))+\frac{A_1}{2T}.
\end{equation}
Then according to theorem \ref{thm approximation}, for any small $ \epsilon_n > 0 $, there is a $ C^1 $ function $ v_n (x) $ that is $ \epsilon_n$-close to $ v(x) $ in $ L^2 [-T, T] $, where $\epsilon_n\xrightarrow{n} 0$, and meets the conditions required by (\ref{integral.v}) and (\ref{endpoints.v}). Similar to $L^1$, denote $\|v(x)\|_{L^2[-(2K_1+1)T,(2K_2+1)T]}\triangleq m$, which is exactly the lower bound of input space for $L^2$ metric. Then
\begin{equation}
\begin{array}{ll}
& \displaystyle |\int_{-\left(2 K_{1}+1\right) T}^{\left(2 K_{2}+1\right) T}|v_n(x)| ^2 d x -m|\\
 \leq & \displaystyle \int_{-T}^{T} |\sum_{i=1}^{K}(t_{i}(x)-v_n(x))^2-\sum_{i=1}^{K}(t_{i}(x)-v(x))^2| d x \\
 \leq & \displaystyle \int_{-T}^{T} |(\sum_{i=1}^K(2t_i(x)-v_n(x)-v(x))(v(x)-v_n(x))| dx \\
 \leq & (\displaystyle \int_{-T}^T (\sum_{i=1}^K(2t_i(x)-v_n(x)-v(x))^2 dx)^{\frac 1 2}\ \|v(x)-v_n(x)\|_{L^2[-T,T]}\leq M\epsilon_n,
\end{array}
\end{equation}
where $\exists M>0$ such that $(\displaystyle \int_{-T}^T (\sum_{i=1}^K(2t_i(x)-v_n(x)-v(x))^2 dx)^{\frac 1 2}\leq M$.

Therefore, TBVP input control problem in $L^2$ metric space can be described in a similar way, though the results differ from that in $L^1$.
Firstly, MS exists if and only if $v(x)$ defined in equation (\ref{vx.L2}) satisfies equation (\ref{endpoints.v}). 
Secondly, if MS exists then it is unique, which is exactly the corresponding solution of such input $v(x)$.
Thirdly, PMS induced by $\lbrace v_n(x) \rbrace$ described above always exists.
Lastly, the limit of PMS always equals to $v(x)$ in (\ref{vx.L2}) almost everywhere in $L^2$, so the uniqueness can also be illustrated in this sense. 

\section{Discussion}

TBVP problem in one-dimensional wave equation has already been solved. In fact, the point that there are infinite solutions is intuitive. According to the nature of wave propagation, the decision area of each point $ (T, x) $ passed in the propagation is a triangle. Returning to $ t = 0 $, it is actually the initial value $f_0(x)$ and $ v (x)$ on the interval $ [x-T, x+T] $ on the $ x $ axis that determines the state at the point $ (T, x) $. In other words, points on an interval determine the point $ (T, x) $, and points on a larger interval determine points on a smaller interval. So we can intuitively feel that in such a decision mode, its solution is not unique.

In this article we mainly discuss the TBVP input control problem for one-dimensional wave equation on a compact supported set $[-(2K_1+1)T, (2K_2+1)T],\ K_1,K_2\in \mathbb N^*$ in $L^p\ (p=1,2)$ metric spaces. By controlling the input in $L^p[-(2K_1+1)T,(2K_2+1)T]$, discussion can be confined to $v(x)$ on $[-T,T]$. Certain conditions ensure the corresponding $u(t,x)$ is $C^2$. Though theorem \ref{thm input} does not use $F$ appeared in (\ref{eq.v}), the recursive relation of $v$ is actually found by $F$. In fact, $F$ satisfying (\ref{eq.F.rec}) is equivalent to $v$ satisfying the recursive relation (\ref{eq.v.rec}) and the integral condition (\ref{integral}). The integral condition for $v(x)$ is due to the loss of information during the differentiation of $F$ and it ensures the continuity of $F$. Additionally, the endpoints relation ensures that $v(x)$ is $C^1$ at the connection points $(2k+1)T,\ k\in \mathbb Z$.

In $L^1$ and $L^2$, there always exists $v(x)$ satisfying the integral condition, which makes the infimum of $\|v(x)\|_{L^p[-(2K_1+1)T,(2K_2+1)T]}$. In $L^2$, such $v(x)$ is $C^{\infty}[-T,T]$. In $L^1$, there is also some $v(x)\in C^{\infty}([-T,T]\backslash P)$, where the Lebesgue measure of $P$ is zero. Using  the denotation of the trapezoidal region $\Omega$, the corresponding $u(t,x)$ could be at least a pseudo-MS, which means
\begin{enumerate}
\item $u(t,x)\in C(\Omega)$.
\item $\exists D\subset\Omega$ and its Lebesgue measure is zero such that $u(t,x) \in C^2(\Omega\backslash D)$ and it satisfies the equation of TBVP in $\Omega\backslash D$.
\item $\forall \tilde u(t,x)\in C^2(\Omega)$ that is the solution of TBVP, let 
\begin{equation}
\Gamma=\{\{(t,x):t=0,x\in [-(2K_1+1)T,(2K_2+1)T]\}\backslash D\},
\end{equation}
then $||u_t(0,x)||_\Gamma\leq||\tilde u_t(0,x)||_\Gamma$.
\end{enumerate}

In the future, following intriguing problems remain open and call for investigation: (1) what happens if we consider wave equations in higher dimensions? (2) what happens if we investigate the input control in more general $L^p$ space, where $1\leq p\leq \infty$, and is there some common phenomena? (3) what happens if we change this simplest wave equation to nonlinear and quasilinear wave equations?

\subsection{Conclusions}
In summary, a real MS exists if and only if there is some pseudo-MS whose $v(x)$ is $C^1[-T,T]$ and satisfies endpoints relation. Therefore, we give an approximation method which could keep the integral, adjust endpoints values of $[-T,T]$ and $C^1$-dense in $L^p$. So here comes a sequence of functions $\lbrace v_n(x)\rbrace $ which makes the solution $C^2$ and could be arbitrarily $\epsilon$-close to the infimum in the metric, so that PMS always exists. Additionally, we find the inputs of MS of $L^1$ all lie on the same order of curve or strip, and that of $L^2$ is always unique if MS exists. Lastly, the uniqueness of PMS is described by its limit in a similar way.

\section{Acknowledgments}

The authors thank Yan Zhang for his useful comments on the manuscript, and Qingyou Sun, Chaojun Yu for their valuable discussions.


\end{document}